\documentclass[11pt]{article}

\usepackage[T1]{fontenc}
\usepackage[utf8]{inputenc}
\usepackage{lmodern}
\usepackage{microtype}

\usepackage{amsmath,amssymb,amsthm}

\usepackage[margin=1in]{geometry}
\usepackage{booktabs,array,tabularx}
\usepackage[table]{xcolor}
\usepackage[shortlabels]{enumitem}
\usepackage{graphicx}
\usepackage{tikz}
\usetikzlibrary{arrows.meta,positioning}

\usepackage{xurl}
\usepackage[colorlinks=true,linkcolor=blue!45!black,
  citecolor=blue!45!black,urlcolor=blue!45!black]{hyperref}
\usepackage[nameinlink,capitalize]{cleveref}

\definecolor{ink}{HTML}{17324D}
\definecolor{sea}{HTML}{2F6F8F}
\definecolor{rust}{HTML}{A45542}
\definecolor{pale}{HTML}{EEF4F6}

\newtheorem{theorem}{Theorem}[section]
\newtheorem{lemma}[theorem]{Lemma}

\newcommand{\D}{\mathbb D}
\newcommand{\Sstar}[1]{\mathcal S^*(#1)}

\hypersetup{
  pdftitle={The Geometric Function Atlas: Using What We Know About Radius
    Problems to Solve New Ones},
  pdfauthor={Prasanna Devadiga, Kishan Gurumurthy, Arya Suneesh, Pushparaj
    Devadiga, Asha Sebastian}}

\title{The Geometric Function Atlas: Using What We Know About Radius
  Problems to Solve New Ones}
\author{Prasanna Devadiga\thanks{These authors contributed equally.}\textsuperscript{,\,1},
Kishan Gurumurthy\footnotemark[1]\textsuperscript{,\,1},
Arya Suneesh\footnotemark[1]\textsuperscript{,\,1},\\
Pushparaj Devadiga\footnotemark[1]\textsuperscript{,\,2},
Asha Sebastian\textsuperscript{1}\\
\normalsize\textsuperscript{1}Department of Computer Science and Engineering,\\
Indian Institute of Information Technology Kottayam, Kerala, India\\
\normalsize\textsuperscript{2}K.~J. Somaiya College of Engineering, Mumbai, India\\
\small\texttt{devadigaprasanna28@gmail.com}, \texttt{kishangurumurthy@outlook.com},\\
\small\texttt{aryasuneesh3@gmail.com}, \texttt{pdevadiga451@gmail.com},
\texttt{asha@iiitkottayam.ac.in}\\
\small Corresponding author: Prasanna Devadiga, \texttt{devadigaprasanna28@gmail.com}}
\date{Word count: approximately 3,900 (main text, excluding references)}

\begin{document}
\maketitle

\begin{abstract}
Usually, radius problems between Ma--Minda classes are solved one pair at a
time, by fitting a disk inside the target region, providing a sufficient but
not always the sharp radius. The same class appearing under different names
makes comparison of results harder, with mistakes and duplication eluding
detection.

The Geometric Function Atlas catalogs radius problems using exact generators
and a direction (source $\to$ target), ensuring unique identification. For
each problem, we then compute the radius numerically, together with the point
where the image first meets the target boundary. This helps us identify the
sharp estimate which guides the proof we write.

With this system, we were able to refine two earlier stated estimates, the
disk radius for the sine $\to$ modified sigmoid and sine $\to$ rational cases
from $\operatorname{arsinh} c$ to $\arcsin c$, and correct earlier stated
values for the crescent and exponential classes: the sharp radius is $\sin 1$
from crescent to exponential, and $\log(1+\sqrt2)$ from exponential to
crescent. All four proofs follow a common strategy: an estimate that is exact
at the contact point, followed by the standard extremal function. Applying the
method, we derive fifteen further sharp radii.
\end{abstract}



\section{Introduction}\label{sec:intro}

Let $\D$ be the unit disk and $\mathcal A$ the class of analytic functions
$f(z)=z+a_2z^2+\cdots$ on $\D$. Ma and Minda~\cite{MaMinda} showed that many
subclasses of starlike functions share one form,
\begin{equation}\label{eq:maminda}
 \Sstar{\varphi}=\Bigl\{f\in\mathcal A:\ \frac{zf'(z)}{f(z)}\prec\varphi(z)\Bigr\},
\end{equation}
where $\varphi$ is univalent with positive real part, $\varphi(\D)$ is
starlike with respect to $1$ and symmetric about the real axis, $\varphi(0)=1$
and $\varphi'(0)>0$. Each choice of $\varphi$ fixes a region $\varphi(\D)$,
and new regions have appeared at a growing rate, especially since 2015: the
lemniscate~\cite{SokolStankiewicz}, exponential~\cite{Mendiratta},
cardioid~\cite{SharmaCardioid}, crescent~\cite{RainaSokol,GandhiLune},
sine~\cite{ChoSine}, Bell~\cite{BellClass}, nephroid~\cite{WaniNephroid},
lima\c{c}on~\cite{MasihKanas}, hyperbolic-cosine~\cite{MundaliaKumar},
modified-sigmoid~\cite{GoelKumar} and
rational~\cite{KumarRavichandranRational} classes, among others. Every new
class raises the same questions, and one of them is the inclusion radius: how
small must a disk be before every function of one class behaves, on that disk,
like a function of another?

Anyone who works on these questions knows the difficulties. The results are
spread over many papers. The same class appears under different names, and the
same name is sometimes used for different classes. Because the radius from one
class to a second is in general not the radius from the second back to the
first, $n$ classes already give $n(n-1)$ questions, and it is hard to tell
which have been answered. The usual proof encloses the image of the disk in a
smaller disk that fits inside the target. This gives a radius that is
sufficient but often not sharp, and it rarely says where the extremal function
touches the boundary.

Examples from this paper show several of these difficulties. The class
generated by $z+\sqrt{1+z^2}$ is called the \emph{crescent} class
in~\cite{RainaSokol} and the \emph{lune} class in~\cite{GandhiLune}. The
radius from the sine class to the modified-sigmoid class has been stated as
$\operatorname{arsinh}\frac{e-1}{e+1}=0.4471$~\cite{GoelKumarRadii}; its sharp
value is $\arcsin\frac{e-1}{e+1}=0.4804$ (\cref{thm:sine-sigmoid}), and the
same gap between $\operatorname{arsinh}$ and $\arcsin$ occurs for the rational
target (\cref{thm:poly}). Between the crescent and exponential classes, both
directions have published values that are too small: $0.4988$, stated as
sharp, from crescent to exponential~\cite{NazKumarRavichandran}, where the
sharp value is $\sin1=0.8415$, and $0.188$ from exponential to
lune~\cite{GandhiLune}, where it is $\log(1+\sqrt2)=0.881$
(\cref{thm:reciprocal}). The value $0.188$ has since been reprinted in a
survey table~\cite{WaniRadius}.

In each case the fix came from the same order of work. We first pinned the
problem down exactly, then computed the radius and the point where the image
first meets the target boundary, and only then wrote the proof. The contact
point tells one which estimate is exact, and hence which proof will be sharp.
This paper describes that order of work and what it produced.

\begin{enumerate}[leftmargin=*,itemsep=2pt]
 \item \textbf{A record of each question by its exact functions and direction}
 (\cref{sec:keys}). The Geometric Function Atlas~\cite{GFAtlasSoftware} is a
 curated, checkable record of radius and coefficient problems for Ma--Minda
 classes. Each record fixes a problem by its exact generators and direction,
 and stores the computed answer, the contact point, the proof status and the
 published statements it was compared with. The Atlas holds 702 directed
 radius questions between 28 source and 26 target classes; class names are
 kept only to help search. It can be browsed at
 \url{https://gft-registry.fly.dev} and rechecked locally with the Python
 package \texttt{geometric-function-atlas}, which replays its certificates.
 \item \textbf{A separation between finding a value and proving it}
 (\cref{sec:pipeline}). Each question passes through five recorded stages, and
 each stage states what it does not yet establish.
 \item \textbf{Four published radii corrected or sharpened, by one recipe}
 (\cref{sec:radii}). Each proof is an estimate that is exact at the contact
 point the computation found (a coefficient majorant or a sine-ray bound, or,
 for the rational target, an interval check along the image curve), followed
 by the standard extremal function. The same recipe gives fifteen further
 sharp radii; four recover known results and serve as checks.
\end{enumerate}

The approach borrows from two traditions. Databases such as the
OEIS~\cite{SloaneOEIS} and the LMFDB~\cite{CremonaLMFDB} show that giving
objects a fixed identity with provenance makes comparison routine.
Experimental mathematics~\cite{BorweinBailey,BorweinBaileyNotices} treats
high-precision computation and the recognition of closed forms~\cite{PSLQ} as
steps toward a theorem that are kept separate from its proof.

A preprint of this paper is~\cite{GFAtlasPreprint}; its first version is
longer and also treats coefficient problems.

\section{Pinning down the question}\label{sec:keys}

For generators $\varphi_1,\varphi_2$ we write
\begin{equation}\label{eq:radius}
 R(\varphi_1\to\varphi_2)
 =\sup\Bigl\{r\le1:\ \tfrac{f(rz)}{r}\in\Sstar{\varphi_2}
 \ \text{for every } f\in\Sstar{\varphi_1}\Bigr\}.
\end{equation}
If $zf'/f=\varphi_1\circ\omega$ with $\omega$ a Schwarz function, then
$|\omega(z)|\le|z|$, and the question becomes when
\begin{equation}\label{eq:containment}
 \varphi_1(\D_r)\subseteq\varphi_2(\D).
\end{equation}
Let $\varphi_2^{-1}$ denote the branch of the inverse with
$\varphi_2^{-1}(1)=0$, and put $\psi=\varphi_2^{-1}\circ\varphi_1$ wherever
this is analytic on $\D_r$. Then \eqref{eq:containment} holds exactly when
$\max_{|u|\le r}|\psi(u)|\le1$. The function
\begin{equation}\label{eq:extremal}
 f_{\varphi_1}(z)=z\exp\int_0^z\frac{\varphi_1(t)-1}{t}\,dt,\qquad
 \frac{zf_{\varphi_1}'(z)}{f_{\varphi_1}(z)}=\varphi_1(z),
\end{equation}
belongs to $\Sstar{\varphi_1}$, and it shows that the containment radius is
also the inclusion radius.

\begin{lemma}[Sharpness]\label{lem:sharp}
Suppose $\max_{|u|\le r_*}|\psi(u)|=1$, the maximum being attained at a point
$u_0$ with $|u_0|=r_*$. Then $R(\varphi_1\to\varphi_2)=r_*$.
\end{lemma}

\begin{proof}
For $r\le r_*$ the containment \eqref{eq:containment} holds, so the radius is
at least $r_*$. Since $|\psi(u_0)|=1$, the point $\varphi_1(u_0)$ lies on
$\partial\varphi_2(\D)$ and so not in the open set $\varphi_2(\D)$. For
$r>r_*$ the dilation $f_r(z)=f_{\varphi_1}(rz)/r$ satisfies
$zf_r'(z)/f_r(z)=\varphi_1(rz)$, which takes the value $\varphi_1(u_0)$ at
$z=u_0/r\in\D$. Hence $f_r\notin\Sstar{\varphi_2}$.
\end{proof}

\Cref{tab:classes} lists the generators used below. We use each name for
exactly one generator.

\begin{table}[!htb]
\centering\small
\caption{Generators used in this paper.}
\label{tab:classes}
\begin{tabular}{@{}lll@{}}
\toprule
Name & Generator $\varphi$ & Introduced in\\
\midrule
starlike & $\varphi_0=(1+z)/(1-z)$ & classical\\
order $\frac12$, order $\frac34$ & $1/(1-z)$,\ \ $(1-z/2)/(1-z)$ & classical\\
lemniscate & $\varphi_L=\sqrt{1+z}$ & \cite{SokolStankiewicz}\\
exponential & $\varphi_e=e^z$ & \cite{Mendiratta}\\
crescent (lune) & $\varphi_C=z+\sqrt{1+z^2}$ & \cite{RainaSokol,GandhiLune}\\
sine & $\varphi_S=1+\sin z$ & \cite{ChoSine}\\
Bell & $\varphi_B=\exp(e^z-1)$ & \cite{BellClass}\\
modified sigmoid & $\varphi_{SG}=2/(1+e^{-z})$ & \cite{GoelKumar}\\
rational & $\varphi_R=1+z(k+z)/(k(k-z))$, $k=1+\sqrt2$ & \cite{KumarRavichandranRational}\\
hyperbolic cosine & $\varphi_H=\cosh\sqrt z$ & \cite{MundaliaKumar}\\
tangent & $\varphi_{\tanh}=1+\tanh z$ & \cite{UllahTanh}\\
\bottomrule
\end{tabular}
\end{table}

\paragraph{Why the direction and the exact generator matter.} A record in the
Atlas stores the ordered pair $(\varphi_1\to\varphi_2)$, the parameter values,
the computed contact point, the closed form if one is found, the proof status,
and the published statement it was compared with. Names, translations and
transliterations are attached as search aliases and play no part in deciding
whether two records describe the same question. Both choices are necessary.
Among the records whose closed form passed a symbolic consistency check, the
two radii differ in 253 of the 262 class pairs recorded in both directions
(\cref{fig:reciprocal}), and 197 of the 463 radii different from $1$ are
shared with another question, so a value alone does not identify its problem.

\begin{figure}[!htb]
\centering
\includegraphics[width=0.45\linewidth]{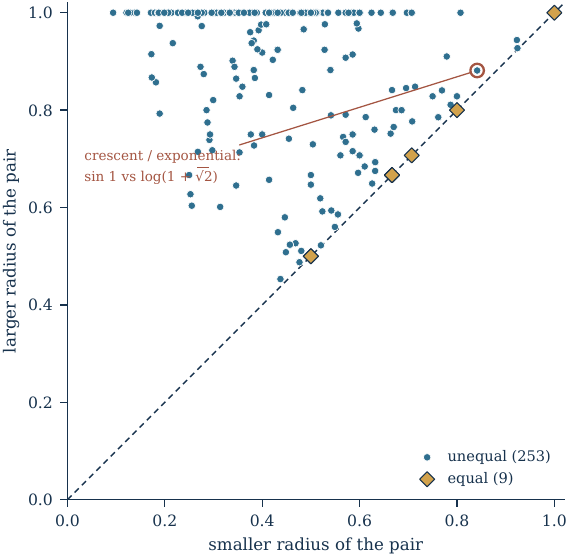}
\caption{Each dot is a class pair recorded in both directions, plotted by its
two radii. Only nine lie on the diagonal. The marked dot is the
crescent--exponential pair of \cref{thm:reciprocal}.}
\label{fig:reciprocal}
\end{figure}

\section{Computing the answer, then proving it}\label{sec:pipeline}

\begin{figure}[!htb]
\centering
\begin{tikzpicture}[font=\footnotesize,>=Latex,node distance=2.5mm]
\tikzset{card/.style={draw=sea,rounded corners,fill=pale,text width=2.45cm,
 minimum height=1.9cm,align=center,inner sep=3pt},arr/.style={-Latex,thick,draw=rust}}
\node[card] (n1) {\textbf{1. Pin down}\\[2pt] sine $\to$ modified sigmoid; generators fixed};
\node[card,right=of n1] (n2) {\textbf{2. Compute}\\[2pt] first contact at\\ $r\approx0.4803810791$, $u=r$};
\node[card,right=of n2] (n3) {\textbf{3. Recognize}\\[2pt] $\arcsin\dfrac{e-1}{e+1}$};
\node[card,right=of n3] (n4) {\textbf{4. Prove}\\[2pt] bound exact at $u=r$, plus extremal};
\node[card,right=of n4] (n5) {\textbf{5. Compare}\\[2pt] sharpens Goel--Kumar's stated radius};
\foreach \x/\y in {n1/n2,n2/n3,n3/n4,n4/n5}{\draw[arr] (\x)--(\y);}
\end{tikzpicture}
\caption{The five stages, shown for the problem of \cref{thm:sine-sigmoid}.}
\label{fig:pipeline}
\end{figure}

\Cref{fig:pipeline} follows one question through the five stages. In the
first, a name is resolved to its generator and its parameter convention, and a
question with an ambiguous class is set aside until its convention is fixed.
In the second, the program increases $r$ and scans the boundary angle until
$\varphi_1(\D_r)$ first touches $\partial\varphi_2(\D)$, and then refines the
radius and the contact point to 60 digits. The contact point is the key
output: a bound gives the sharp radius only if it is exact at that point, so
the contact point decides which proof to write (\cref{sec:radii}). In the
third, the value is matched against a small vocabulary of expressions built
from the two generators: algebraic numbers, logarithms, inverse trigonometric
and hyperbolic functions, and short compositions of these. In the fourth, the
value is proved, and in the fifth it is compared with published statements on
the generators, the direction, the value and the wording of sharpness.

Each stage stops short of a stronger claim, and the Atlas records which one.
\Cref{tab:evidence} lists the states a question can be in.

\begin{table}[!htb]
\centering\small
\caption{What each recorded state requires, and what it does not yet show.}
\label{tab:evidence}
\begin{tabularx}{\linewidth}{@{}l>{\raggedright\arraybackslash}X>{\raggedright\arraybackslash}X@{}}
\toprule
State & Requires & Does not yet show\\
\midrule
Computed & a stable 60-digit radius and contact point & that the value has a closed form\\
Closed form recognized & a closed form agreeing with the value to 60 digits & that the closed form is correct\\
Contact equation verified & the contact equation, checked symbolically & that no other boundary point is reached first\\
Proved & a global bound and an extremal function that meet & that the result is new\\
Known & a published statement of the same directed problem & ---\\
\bottomrule
\end{tabularx}
\end{table}

The first four states are successive. Known is recorded separately, and a
known value may still be corrected (\cref{sec:radii}).
\section{One recipe for sharp radii}\label{sec:radii}

By the maximum principle, $\max_{|u|\le r}|\psi|$ is reached on $|u|=r$, and
the computation says where. A proof is sharp exactly when its estimate is an
equality at that contact point. So the contact point chooses the proof. When
the image first touches the target on the positive real axis and $\psi$ has
nonnegative Taylor coefficients, the bound $|\psi(u)|\le\psi(|u|)$ is exact
there. When the contact is at $u=-r$ or $u=ir$, the same bound applied to the
absolute values of the coefficients is exact if the coefficients alternate in
the matching way. When $\psi$ contains $\sin u$ inside a logarithm or an
inverse hyperbolic tangent, a bound by integration along rays is exact on the
real axis. The two lemmas below state these facts. When the target has no
workable inverse, the proof instead checks a polynomial inequality along the
image curve with interval arithmetic.

\begin{lemma}[Coefficient majorant]\label{lem:majorant}
Let $G(u)=\sum_{n\ge1}g_nu^n$ converge for $|u|<\rho$ and let $0<r<\rho$. Then
$|G(u)|\le\sum_{n\ge1}|g_n|r^n$ for $|u|\le r$, with equality at
$u=re^{i\theta}$ whenever the numbers $g_ne^{in\theta}$ share a common
argument. In particular, if every $g_n\ge0$ then $\max_{|u|\le r}|G|=G(r)$.
\end{lemma}

\begin{lemma}[Sine-ray bounds]\label{lem:sine}
For $|u|\le r<\pi/2$,
\[
 |\operatorname{artanh}(\sin u)|\le\operatorname{artanh}(\sin r),\qquad
 |\log(1+\sin u)|\le-\log(1-\sin r),
\]
with equality at $u=\pm r$ and at $u=-r$ respectively.
\end{lemma}

\begin{proof}
Integrate along the segment from $0$ to $u$, and write $s=|t|$ for points $t$
on it. For the first bound, $\operatorname{artanh}(\sin u)=\int_0^u\sec
t\,dt$, and $|\cos(x+iy)|^2=\cos^2x+\sinh^2y\ge\cos^2 s$ when $x^2+y^2=s^2$,
so $|\sec t|\le\sec s$. For the second, $\frac{d}{dt}\log(1+\sin
t)=\frac{1-\sin t}{\cos t}=\frac{1-\tau}{1+\tau}$ with $\tau=\tan(t/2)$. Since
$\tan$ has nonnegative Taylor coefficients, \cref{lem:majorant} gives
$|\tau|\le\tan(s/2)=:\sigma<1$, and then
$|1-\tau|/|1+\tau|\le(1+\sigma)/(1-\sigma)=\tan(\frac\pi4+\frac s2)=(1+\sin
s)/\cos s$, with equality at $t=-s$. Integrating over $0\le s\le r$ gives both
bounds.
\end{proof}

\Cref{tab:radii} lists the radii we prove, with the contact point the
computation found and the route that closes each one. The first four correct
or sharpen published values; the other fifteen show that the recipe is not
special to them. All but four of the nineteen follow from \cref{lem:majorant}
or \cref{lem:sine} and \cref{lem:sharp}. We chose the fifteen to cover every
route we found; many more records in the Atlas are likely to yield to the same
recipe.

\begin{table}[!htb]
\centering\small
\setlength{\tabcolsep}{4pt}
\caption{Radii proved in this paper. Shaded: published values that we correct
or sharpen. Contact: where $\varphi_1(\D_r)$ first meets
$\partial\varphi_2(\D)$. Route: M = \cref{lem:majorant}, S = \cref{lem:sine},
T = triangle inequality, P = implicit target checked with interval arithmetic.
Earlier work: ``--'' means that we found no earlier statement of the same
directed problem; this is a search result, not a claim of novelty.}
\label{tab:radii}
\begin{tabular}{@{}llllcll@{}}
\toprule
& From $\to$ to & Sharp radius & Value & Contact & Route & Earlier work\\
\midrule
\multicolumn{7}{@{}l}{\emph{Published values corrected or sharpened}}\\
\rowcolor{pale}\ref{thm:sine-sigmoid} & sine $\to$ mod.\ sigmoid & $\arcsin\frac{e-1}{e+1}$ & 0.480381 & $\pm r$ & S & sharpens~\cite{GoelKumarRadii}\\
\rowcolor{pale}\ref{thm:poly}(a) & sine $\to$ rational & $\arcsin(3-2\sqrt2)$ & 0.172426 & $-r$ & P & sharpens~\cite{KumarRavichandranRational}\\
\rowcolor{pale}\ref{thm:reciprocal}(a) & crescent $\to$ exponential & $\sin1$ & 0.841471 & $\pm ir$ & M & corrects~\cite{NazKumarRavichandran}\\
\rowcolor{pale}\ref{thm:reciprocal}(b) & exponential $\to$ crescent & $\log(1+\sqrt2)$ & 0.881374 & $r$ & M & corrects~\cite{GandhiLune}\\
\midrule
\multicolumn{7}{@{}l}{\emph{Further radii from the same recipe}}\\
\ref{thm:maj}(a) & hyp.\ cosine $\to$ lemniscate & $(\operatorname{arsinh}1)^2$ & 0.776819 & $r$ & M & --\\
\ref{thm:maj}(b) & exponential $\to$ sine & $\log(1+\sin1)$ & 0.610565 & $r$ & M & --\\
\ref{thm:maj}(c) & rational $\to$ sine & see \cref{thm:maj}(c) & 0.914906 & $r$ & M & --\\
\ref{thm:maj}(d) & Bell $\to$ sine & $\log(1+\log(1+\sin1))$ & 0.476585 & $r$ & M & --\\
\ref{thm:maj}(e) & exponential $\to$ Bell & $1-e^{-1}$ & 0.632121 & $-r$ & M & --\\
\ref{thm:maj}(f) & exponential $\to$ lemniscate & $\frac12\log2$ & 0.346574 & $r$ & M & --\\
\ref{thm:maj}(g) & exponential $\to$ order $\frac12$ & $\log2$ & 0.693147 & $-r$ & M & known~\cite{Mendiratta,KhatterExpLemniscate}\\
\ref{thm:maj}(h) & starlike $\to$ lemniscate & $3-2\sqrt2$ & 0.171573 & $r$ & M & known~\cite{AliJainRavichandran,SebastianRavichandranSub}\\
\ref{thm:maj}(i) & starlike $\to$ order $\frac34$ & $\frac17$ & 0.142857 & $-r$ & M & known~\cite{GangadharanRavichandranShanmugam,SebastianRavichandranSub}\\
\ref{thm:sin}(a) & sine $\to$ tangent & $\arcsin(\tanh1)$ & 0.865769 & $\pm r$ & S & --\\
\ref{thm:sin}(b) & sine $\to$ exponential & $\arcsin(1-e^{-1})$ & 0.684287 & $-r$ & S & --\\
\ref{thm:sin}(c) & sine $\to$ Bell & $\arcsin(1-e^{1/e-1})$ & 0.487633 & $-r$ & S+M & --\\
\ref{thm:tri}(a) & crescent $\to$ lemniscate & $\frac{\sqrt2}{4}$ & 0.353553 & $r$ & T & --\\
\ref{thm:tri}(b) & order $\frac12\to$ crescent & $2-\sqrt2$ & 0.585786 & $r$ & T & known~\cite{KrishnanRavichandranSaikrishnan}\\
\ref{thm:poly}(b) & mod.\ sigmoid $\to$ rational & $\frac12\log2$ & 0.346574 & $-r$ & P & --\\
\bottomrule
\end{tabular}
\end{table}

\subsection{A disk bound and its sharp replacement}

We give the first proof in full, since the others follow its pattern.

\begin{theorem}\label{thm:sine-sigmoid}
$R(\varphi_S\to\varphi_{SG})=\arcsin\dfrac{e-1}{e+1}=0.480381\ldots$
\end{theorem}

\begin{proof}
The inverse of the target undoes the sigmoid:
$\varphi_{SG}^{-1}(w)=\log\frac{w}{2-w}$, so $\psi(u)=\log\frac{1+\sin
u}{1-\sin u}=2\operatorname{artanh}(\sin u)$. The branch is the one fixed at
$0$; it is valid because $|\sin u|\le\sinh r<1$ for the radii in question, and
$\varphi_{SG}$ is univalent on $\D$ because $e^z$ is injective on
$|\operatorname{Im}z|<\pi$. By \cref{lem:sine}, $\max_{|u|\le
r}|\psi(u)|=2\operatorname{artanh}(\sin r)$, attained at $u=r$. This equals
$1$ when $\sin r=\tanh\frac12=\frac{e-1}{e+1}$, and \cref{lem:sharp} completes
the proof.
\end{proof}

The earlier argument~\cite[Thm.~2.9(ii)]{GoelKumarRadii} uses $|\sin
u|\le\sinh r$ and a lemma on disks inside the sigmoid region. It states
$\operatorname{arsinh}\frac{e-1}{e+1}=0.447074\ldots$ as the radius, but the
proof establishes only that this radius is sufficient, and it gives no
extremal function.\footnote{We compared with arXiv:2208.01241v1. On 24
September 2026 this was the only arXiv version, zbMATH Open held only the
preprint record, and Crossref listed no journal version.} The estimate $|\sin
u|\le\sinh r$ is an equality at $u=\pm ir$, where the sine image is widest,
but the image first meets the sigmoid boundary on the real axis
(\cref{fig:sine-sigmoid}). The sharp radius is larger by 7.45\%.

The same step appears for the rational target. Kumar and
Ravichandran~\cite[Thm.~3.6(v)]{KumarRavichandranRational} give
$\operatorname{arsinh}(3-2\sqrt2)=0.170742\ldots$ for the rational radius of
the sine class, by the same disk argument, and this value is restated in the
survey table of~\cite{WaniRadius}. The sharp value is
$\arcsin(3-2\sqrt2)=0.172426\ldots$ (\cref{thm:poly}(a)). In both cases the
disk estimate replaces the real contact by an imaginary one, and the correct
answer replaces $\operatorname{arsinh}$ by $\arcsin$.

\begin{figure}[!htb]
\centering
\includegraphics[width=0.55\linewidth]{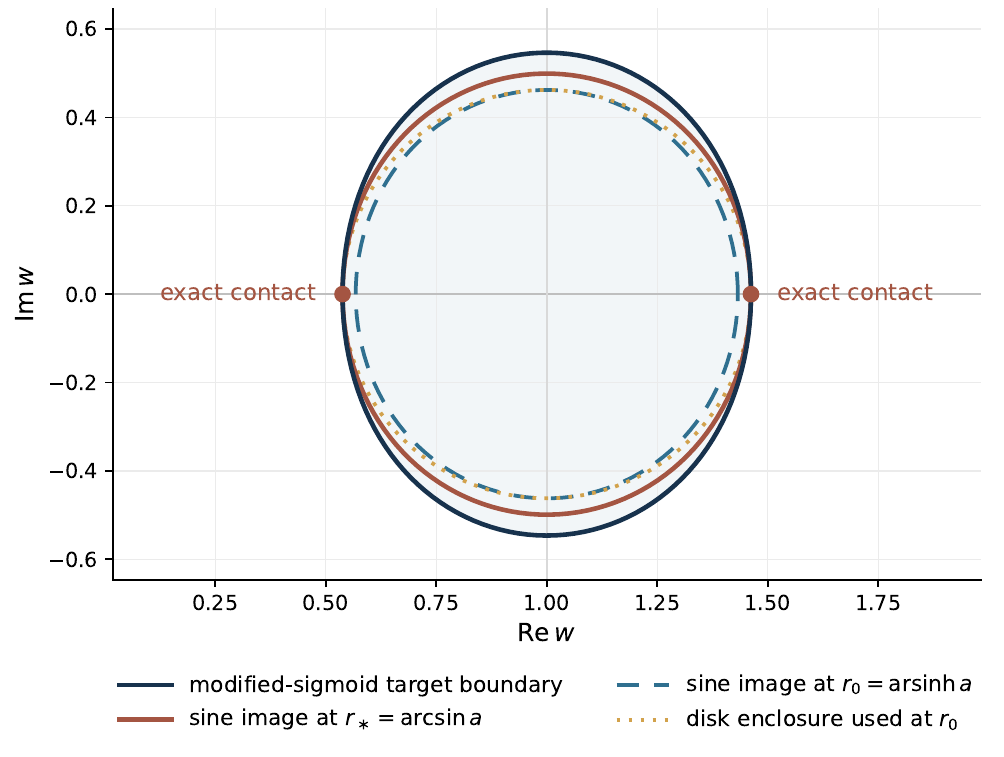}
\caption{Sine to modified sigmoid, in the plane of $w=zf'/f$. Dotted: the disk
behind the stated radius $0.4471$. Solid: the image at the sharp radius
$0.4804$, which first meets the target on the real axis.}
\label{fig:sine-sigmoid}
\end{figure}

\subsection{A reciprocal pair with two corrected constants}\label{sec:reciprocal}

\begin{theorem}\label{thm:reciprocal}
For the crescent and exponential classes,
\[
 R(\varphi_C\to\varphi_e)=\sin1=0.841470\ldots \qquad\text{and}\qquad
 R(\varphi_e\to\varphi_C)=\log(1+\sqrt2)=0.881373\ldots
\]
\end{theorem}

\begin{proof}
(a) On the branches fixed at $0$,
$\varphi_e^{-1}(\varphi_C(u))=\log(u+\sqrt{1+u^2})=\operatorname{arsinh}u$.
Its coefficients alternate in sign on odd powers, and their absolute values
sum to $\arcsin r$. At $u=ir$ every term $g_n(ir)^n$ is a positive multiple of
$i$, so \cref{lem:majorant} gives $\max_{|u|\le
r}|\operatorname{arsinh}u|=\arcsin r$, attained at $u=ir$. The radius is
$r=\sin1$, where $\varphi_C(i\sin1)=\cos1+i\sin1=e^{i}$ lies on the boundary
of $\varphi_e(\D)$.

(b) $\varphi_C^{-1}(e^u)=\sinh u$ on the right-hand branch, since
$\operatorname{Re}\cosh u>0$ for $|u|<r$. Its coefficients are nonnegative, so
the radius solves $\sinh r=1$, with equality at $u=r$, where
$e^r=1+\sqrt2=\varphi_C(1)$ is the tip of the crescent.

\Cref{lem:sharp} gives sharpness in both directions.
\end{proof}

\begin{figure}[!htb]
\centering
\includegraphics[width=0.45\linewidth]{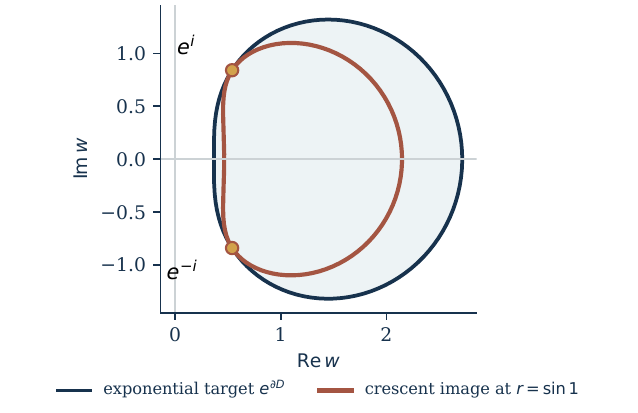}\hfill 
\includegraphics[width=0.46\linewidth]{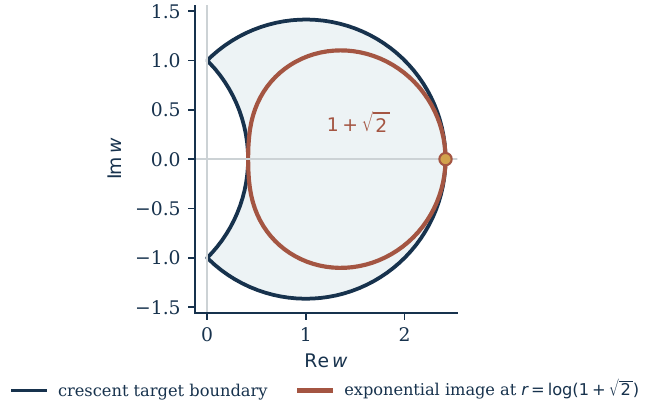}
\caption{The two directions of \cref{thm:reciprocal}. Left: at $r=\sin1$ the
crescent image first touches the exponential region at $e^{\pm i}$, off the
real axis. Right: at $r=\log(1+\sqrt2)$ the exponential image touches the
crescent at its tip on the real axis.}
\label{fig:reciprocal-pair}
\end{figure}

The two directions reach the boundary at different kinds of points
(\cref{fig:reciprocal-pair}), which is why their radii differ.

Both directions have appeared in print with values that are too small. Naz,
Kumar and Ravichandran~\cite[Thm.~4.5(b)]{NazKumarRavichandran} give
$0.498824\ldots$ as the exponential radius of the crescent class and state
that it is sharp. Their argument fits a disk inside the region $|\log w|<1$,
and the disk estimate is not attained by the crescent image;
\cref{thm:reciprocal}(a) shows that the containment holds up to
$\sin1=0.841\ldots$. The second direction appears under the other name of the
class. Gandhi and Ravichandran~\cite[Thm.~3.5]{GandhiLune} give
$\log((\sqrt2+1)/2)=0.188226\ldots$ as the lune radius of the exponential
class. Their proof of part (c) uses $|e^z-1|\le1-e^{-r}$ on $|z|=r$, which
fails at $z=r$ because $e^r-1>1-e^{-r}$. Even the valid disk estimate
$|e^z-1|\le e^r-1$ gives the larger sufficient radius
$\log(3-\sqrt2)=0.461\ldots$, and \cref{thm:reciprocal}(b) gives the sharp
value $0.881\ldots$. This value is repeated in the survey table
of~\cite{WaniRadius}.

The remaining subsections prove the other radii of \cref{tab:radii}, grouped
by route.

\subsection{Coefficient-majorant proofs}

\begin{theorem}\label{thm:maj}
The following radii are sharp.
\begin{enumerate}[(a),leftmargin=*,itemsep=1pt]
 \item
 $R(\varphi_H\to\varphi_L)=(\operatorname{arsinh}1)^2=(\operatorname{arcosh}\sqrt2)^2$.
 \item $R(\varphi_e\to\varphi_S)=\log(1+\sin1)$.
 \item
 $R(\varphi_R\to\varphi_S)=\dfrac{1+\sqrt2}{2}\Bigl(\sqrt{1+6\sin1+\sin^21}-1-\sin1\Bigr)$.
 \item $R(\varphi_B\to\varphi_S)=\log(1+\log(1+\sin1))$.
 \item $R(\varphi_e\to\varphi_B)=1-e^{-1}$.
 \item $R(\varphi_e\to\varphi_L)=\frac12\log2$.
 \item $R(\varphi_e\to\varphi_{1/2})=\log2$.
 \item $R(\varphi_0\to\varphi_L)=3-2\sqrt2$.
 \item $R(\varphi_0\to\varphi_{3/4})=\frac17$.
\end{enumerate}
\end{theorem}

\begin{proof}
Each part follows one pattern: find $\psi=\varphi_2^{-1}\circ\varphi_1$, bound
it with \cref{lem:majorant}, check that the bound is attained at the contact
point, and solve for the radius at which the bound equals $1$.
\Cref{lem:sharp} then gives sharpness.

(a) The lemniscate region is $|w^2-1|<1$, so $\varphi_L^{-1}(w)=w^2-1$ and
$\psi(u)=\cosh^2\sqrt u-1=\sinh^2\sqrt u$. This is an entire function of $u$
with nonnegative Taylor coefficients, so its maximum on $|u|\le r$ is
$\sinh^2\sqrt r$, attained at $u=r$. It equals $1$ when $\sqrt
r=\operatorname{arsinh}1$.

The remaining parts are shorter versions of the same argument.

(b) $\psi(u)=\arcsin(e^u-1)$ is a composition of series with nonnegative
coefficients. The bound equals $1$ when $e^r-1=\sin1$, with equality at $u=r$.

(c) $\varphi_R(u)-1=u/k+(2u^2/k^2)(1-u/k)^{-1}$ has nonnegative coefficients,
hence so does $\psi(u)=\arcsin(\varphi_R(u)-1)$. Solving
$\varphi_R(r)-1=\sin1$ gives the stated root, with equality at $u=r$.

(d) $\psi(u)=\arcsin(\exp(e^u-1)-1)$ has nonnegative coefficients. The bound
equals $1$ when $e^r-1=\log(1+\sin1)$, with equality at $u=r$.

(e) $\varphi_B^{-1}(e^u)=\log(1+u)$, whose coefficients $(-1)^{n+1}/n$ give
the bound $-\log(1-r)$, with equality at $u=-r$. It equals $1$ at
$r=1-e^{-1}$.

(f) For $|u|<\pi/2$ the value $e^u$ lies in the right half-plane, so
$\psi(u)=e^{2u}-1$, bounded by $e^{2r}-1$ with equality at $u=r$.

(g) $\varphi_{1/2}^{-1}(w)=1-1/w$ and $\psi(u)=1-e^{-u}$, bounded by $e^r-1$
with equality at $u=-r$.

(h) $\psi(u)=\varphi_0(u)^2-1=4u/(1-u)^2$ has nonnegative coefficients. The
bound equals $1$ when $4r=(1-r)^2$, whose root in $(0,1)$ is $3-2\sqrt2$; the
extremal function is the Koebe function.

(i) $\varphi_{3/4}^{-1}(w)=(w-1)/(w-\frac12)$ and $\psi(u)=4u/(1+3u)$, bounded
by $4r/(1-3r)$ with equality at $u=-r$. The bound equals $1$ at $r=1/7$.
\end{proof}

Parts (g)--(i) are known and serve as checks that the recipe recovers known
answers: (g) is~\cite[Thm.~3.1(i)]{Mendiratta}
(also~\cite[Thm.~2.13(1)]{KhatterExpLemniscate}), (h)
is~\cite[Thm.~2.5]{AliJainRavichandran}, and (i) is a special case
of~\cite{GangadharanRavichandranShanmugam}; (h) and (i) are restated
in~\cite[Lemma~1.2]{SebastianRavichandranSub}.

\subsection{Sine-ray proofs}

\begin{theorem}\label{thm:sin}
The following radii are sharp.
\begin{enumerate}[(a),leftmargin=*,itemsep=1pt]
 \item $R(\varphi_S\to\varphi_{\tanh})=\arcsin(\tanh1)$.
 \item $R(\varphi_S\to\varphi_e)=\arcsin(1-e^{-1})$.
 \item $R(\varphi_S\to\varphi_B)=\arcsin(1-e^{1/e-1})$.
\end{enumerate}
\end{theorem}

\begin{proof}
(a) Here $\psi(u)=\operatorname{artanh}(\sin u)$, and \cref{lem:sine} shows
that $|\psi|\le1$ exactly up to $\sin r=\tanh1$. (b) Here $\psi(u)=\log(1+\sin
u)$, and \cref{lem:sine} gives $|\psi|\le-\log(1-\sin r)$, which equals $1$
when $\sin r=1-e^{-1}$, with equality at $u=-r$. (c) Here
$\psi(u)=\log(1+L(u))$ with $L(u)=\log(1+\sin u)$. Combining \cref{lem:sine}
with the bound $|\log(1+v)|\le-\log(1-|v|)$ used in \cref{thm:maj}(e) gives
$|\psi(u)|\le-\log(1+\log(1-\sin r))$, with equality at $u=-r$; this equals
$1$ when $\sin r=1-e^{1/e-1}$. \Cref{lem:sharp} applies in each case.
\end{proof}

\subsection{Triangle-inequality proofs}

\begin{theorem}\label{thm:tri}
The following radii are sharp.
\begin{enumerate}[(a),leftmargin=*,itemsep=1pt]
 \item $R(\varphi_C\to\varphi_L)=\sqrt2/4$.
 \item $R(\varphi_{1/2}\to\varphi_C)=2-\sqrt2$.
\end{enumerate}
\end{theorem}

\begin{proof}
(a) $\psi(u)=(u+\sqrt{1+u^2})^2-1=2u\bigl(u+\sqrt{1+u^2}\bigr)$. On $|u|=r$
the triangle inequality gives $|\psi|\le2r(r+\sqrt{1+r^2})$, with equality at
$u=r$. Setting this equal to $1$ gives $2r\sqrt{1+r^2}=1-2r^2$, hence
$8r^2=1$. The extremal function reaches $\varphi_C(r)=\sqrt2$, the tip of the
lemniscate.

(b) $\varphi_C^{-1}(w)=(w^2-1)/(2w)$ and $\psi(u)=(2u-u^2)/(2(1-u))$. With
$u=re^{it}$ and $c=\cos t$, the condition $|\psi|\le1$ reads
$(r^4-4)+4rc(2-r^2)\le0$, which is increasing in $c$ for $0<r<\sqrt2$. At
$c=1$ it becomes $F(r)=r^4-4r^3+8r-4\le0$. Since $F(0)=-4$, $F$ increases on
$(0,2-\sqrt2)$ and $F(2-\sqrt2)=0$, the radius is $2-\sqrt2$. The extremal
function $z/(1-z)$ reaches $1+\sqrt2=\varphi_C(1)$, the tip of the crescent.
\end{proof}

Part (b) is known and is a further check: \cite{GandhiLune} gives $2-\sqrt2$
for the convex class, with the same extremal function $z/(1-z)$,
and~\cite[Thm.~7.3]{KrishnanRavichandranSaikrishnan} gives it for a class that
coincides with $\mathcal S^*(\frac12)$.

\subsection{When the target has no convenient inverse}

The inverse of the rational generator involves a square root whose branch is
awkward on the whole image, so for this target we work with its boundary
curve. Put $\rho=\sqrt2-1$ and write $A=w-1=x+iy$, $q=x^2+y^2$. Eliminating
$t$ from $A=t(1+t)/(1-t)$, $|t|=\rho$, gives
\begin{equation}\label{eq:Fpoly}
 F_\rho(x,y)=-\rho^8+\rho^6(q+2x+1)+\rho^4(2xq+6q+2x)+\rho^2(q^2+2xq+q)-q^2,
\end{equation}
and the target is the component of $\{F_\rho\ge0\}$ that contains $A=0$.

\begin{theorem}\label{thm:poly}
The following radii are sharp.
\begin{enumerate}[(a),leftmargin=*,itemsep=1pt]
 \item $R(\varphi_S\to\varphi_R)=\arcsin(3-2\sqrt2)$.
 \item $R(\varphi_{SG}\to\varphi_R)=\frac12\log2$.
\end{enumerate}
\end{theorem}

\begin{proof}
(a) At $r=\arcsin(3-2\sqrt2)$ the point $u=-r$ maps to $A=-(3-2\sqrt2)$, the
left end of the target. Along $u=re^{i\theta}$, outward-rounded interval
evaluation~\cite{Rump} of $F_\rho(\operatorname{Re}\sin
u,\operatorname{Im}\sin u)$ on 2{,}000 subintervals of $[0,\pi-\frac1{20}]$
gives the lower bound $9.06\times10^{-6}$. On $[\pi-\frac1{20},\pi]$ the value
and the first $\theta$-derivative vanish at $\theta=\pi$ and the second
derivative is at least $0.0131$. So the image stays in the target and touches
it only at $u=-r$.

(b) Here $\varphi_{SG}(u)-1=\tanh(u/2)$, which equals $-(3-2\sqrt2)$ at
$u=-\frac12\log2$. The same procedure on $[0,\pi-\frac1{100}]$ gives the lower
bound $8.30\times10^{-8}$, and the second derivative on
$[\pi-\frac1{100},\pi]$ is at least $0.0133$.

\Cref{lem:sharp} gives sharpness in both cases.
\end{proof}

\section{Checks on the method}\label{sec:checks}

A record of this kind is only useful if it returns known answers when the
answers are known. Besides the four known radii of \cref{tab:radii}, two
further checks test this.

Five records can be compared directly with the sigmoid-radius statements of
Goel and Kumar~\cite{GoelKumarRadii}. Four agree with their values and theorem
numbers; the fifth is the sharpening of \cref{thm:sine-sigmoid}. As a further
check we entered two deliberately wrong versions of a known record, one with
the direction reversed and one with the value changed by $10^{-3}$; neither
was matched to the published statement.

\section{What is open, and how to extend the Atlas}\label{sec:open}

\Cref{tab:atlas} lists the status of the 702 questions, grouped according to
whether the first point of contact of the source image with the target
boundary lies on the real axis or elsewhere. Whether the contact point lies on
the real axis or not determines where an estimate must be exact for the proof
to be sharp, and so largely how hard the problem is. The 702 questions break
down as follows: 19 cases proved in \cref{sec:radii}, 391 unfinished cases in
which the image touches the boundary first on the real axis, 138 unfinished
cases in which it touches first elsewhere, 12 cases that failed a symbolic
consistency check, and 142 cases in which the radius is 1.

\begin{table}[!htb]
\centering\small
\caption{The 702 directed radius questions in the Atlas, grouped by where the
source image first meets the target boundary (radius snapshot released with
\texttt{geometric-function-atlas} 0.4.0).}
\label{tab:atlas}
\begin{tabularx}{\linewidth}{@{}>{\raggedright\arraybackslash}p{4.6cm}r>{\raggedright\arraybackslash}X>{\raggedright\arraybackslash}X@{}}
\toprule
Group & Questions & Known & Still missing\\
\midrule
Proved in \cref{sec:radii} & 19 & the sharp radius & ---\\
Real-axis contact & 391 & the value, from $\varphi_1(\pm r)=\varphi_2(\pm1)$ & a proof that the axis is reached first\\
\quad contact equation verified & 306 & & \\
\quad closed form recognized & 79 & & \\
\quad missed by the search & 6 & & \\
Off-axis contact & 138 & a 60-digit value & a proof\\
\quad closed form recognized & 60 & & \\
\quad no closed form found & 78 & & a closed form\\
Failed consistency check & 12 & --- & a corrected record\\
Radius $1$ & 142 & the first class lies inside the second & nothing\\
\bottomrule
\end{tabularx}
\end{table}

For the 391 unfinished cases in which the contact point lies on the real axis,
the computed radius is the smallest $r$ in $(0,1)$ solving $\varphi_1(r) =
\varphi_2(1)$ or $\varphi_1(-r) = \varphi_2(-1)$. The value itself is
therefore known, although 6 of these values were not recognized by our search.
The remaining question is to give a proof that the image cannot touch the
boundary first away from the real axis. In 348 of the 391 cases, the first 22
coefficients of $\psi$ have the sign pattern that makes \cref{lem:majorant}
exact at the contact point, which suggests that \cref{lem:majorant} would
prove these radii. The remaining 43 cases have mixed signs in the first 22
coefficients. A proof that this pattern holds for all coefficients would
settle these 348 cases.

The 138 unfinished cases in which the image touches the boundary first away
from the real axis are more difficult. Here the real-axis equation gives the
sharp radius in only one case, and the disk bound in none. However, for 60 of
them a closed form agrees with the computed value to 60 digits, though it is
not yet proved, while for the remaining 78 we have only a 60-digit value. As
examples of the latter, the radius from the Bell class to the parabolic class
is $0.6766786851252\ldots$, and the radius from the nephroid class to the
modified-sigmoid class is $0.4959247195262\ldots$. We also mention one case
from the crescent class to the exponential class whose 60-digit value our
search did not recognize, although it equals $\sin 1$;
\cref{thm:reciprocal}(a) proves this value. This shows that failing to find a
closed form does not mean that none exists. Each record also gives the contact
point, and the extremal function is the standard one \eqref{eq:extremal} for
the source class, which together make a concrete basis to attempt a manual
proof.

The software and instructions for reproducing these results are described in
\cref{app:software}.

\paragraph{Extending the Atlas.} Nothing in the method is specific to the
radius problems studied here. The Atlas also records coefficient problems,
such as Fekete--Szeg\H{o} and Hankel-determinant bounds, for the same classes;
these are described in the package documentation. Through the Python interface
of our package, researchers can enter new generators as exact expressions; the
package documentation describes how (\cref{app:software}). Once a generator is
added to the Atlas, directed radius questions to and from every existing class
are created. Then the same five stages of our procedure are applied to the new
questions. Stage 4, the proof, cannot be automated: when an answer lies
outside the vocabulary, or when the target generator has no convenient
inverse, a human must write the proof, as in \cref{thm:poly}.

\section{Conclusion}\label{sec:conclusion}

Radius problems between Ma--Minda classes become much easier once each one is
recorded by its exact functions and direction and its answer is computed
before a proof is attempted. The computed contact point shows which estimate
is exact. With that information, two published disk bounds and two published
constants turn out to need correction, for reasons the computation makes
visible, and all four corrections, along with fifteen further radii, follow
from estimates that are exact at the contact point and the standard extremal
function.

The method has limits. Our literature corpus misses closed-access papers
without an open version, so ``--'' in \cref{tab:radii} records only what we
found. Recognition is limited by its vocabulary, so an unidentified value may
still have a closed form. Interval subdivision stalls when a maximum is
attained on a curve. The certificates are symbolic and interval computations
resting on the classical Ma--Minda reduction; they are not formal proofs in a
proof assistant. The 529 unfinished questions in the Atlas are the natural
next problems.

\section*{Funding}
No funding was received for this work.

\section*{Disclosure statement}
The authors report there are no competing interests to declare.

\section*{Declaration of generative AI use}
The authors used Anthropic's Claude (Opus 5.5) to copyedit author-written text
and improve its clarity and flow. They used OpenAI's Codex (GPT 5.6 Sol) to
assist in developing visualization code and portions of the computational code
underlying the work. The authors reviewed and verified the AI-assisted text
and code and confirmed the originality and accuracy of the manuscript's
content, including its references. They checked the terms of use of both tools
and confirmed that their use and the resulting material are suitable for
publication. The authors take full responsibility for the integrity of the
entire work, including the text, code, figures, computations, conclusions, and
accuracy of references.

\section*{Data availability statement}
The software and the database copy of the Atlas are available at
\url{https://github.com/Prasanna28Devadiga/geometric-function-atlas} and on
PyPI (\url{https://pypi.org/project/geometric-function-atlas/}); the Atlas can
be browsed at \url{https://gft-registry.fly.dev}. An archival DOI will
identify the version of record.

\appendix
\section{Software and reproduction}\label{app:software}

We have released the software as the open-source Python package
\texttt{geometric-\allowbreak function-\allowbreak
atlas}~\cite{GFAtlasSoftware}. On Linux and macOS it can be installed with
\begin{center}\small
\texttt{curl -{}-proto '=https' -{}-tlsv1.2 -LsSf https://gft-registry.fly.dev/install.sh | sh}
\end{center}
(on Windows, \texttt{irm https://gft-registry.fly.dev/install.ps1 | iex}),
which installs Python if needed as well as the command \texttt{gfa}. Then
\texttt{gfa verify-radius-certificate sine sigmoid} replays the certificate
for \cref{thm:sine-sigmoid}, rechecking each step.

Of the 19 radii proved here, 8 have certificates that the package replays
(\cref{thm:sine-sigmoid}, \cref{thm:maj}(f)--(i), \cref{thm:sin}(a) and
\cref{thm:tri}); the other 11 are stored with their values.

Documentation, including a getting-started guide, is at
\url{https://prasanna28devadiga.github.io/geometric-function-atlas/getting-started/}.
The workflow for entering a new generator as an exact expression, and
obtaining the Taylor coefficients of the generator and of its extremal
function \eqref{eq:extremal} together with its Fekete--Szeg\H{o} bounds, is at
\url{https://prasanna28devadiga.github.io/geometric-function-atlas/workflows/custom_class/}.

\end{document}